\documentclass[12pt]{article}
\usepackage{amsfonts}
\usepackage{amssymb,latexsym,amsmath, amsthm}
\usepackage{hyperref}\hypersetup{colorlinks, linkcolor=blue, pdfstartview=FitH}
\usepackage{amscd}
\usepackage[all]{xy}
\usepackage{xcolor}
\usepackage{graphicx}

\newtheorem{corollary}{Corollary}
\newtheorem{lemma}[corollary]{Lemma}
\newtheorem{definition}[corollary]{Definition}

\newtheorem{theorem}[corollary]{Theorem}

\begin{document}

\title{Composition Operators with Quasiconformal Symbols
}
\date{}

\author{{
XIANG FANG, KUNYU GUO, AND ZIPENG WANG}}

\maketitle
\def\cc{\mathbb{C}}
\def\zz{\mathbb{Z}}
\def\nn{\mathbb{N}}
\def\rr{\mathbb{R}}
\def\qq{\mathbb{Q}}
\def\dd{\mathbb{D}}
\def\tt{\mathbb{T}}
\def\bb{\mathbb{B}}
\def\ff{\mathbb{F}}

\def\A{\mathcal{A}}
\def\B{\mathcal{B}}
\def\D{\mathcal{D}}
\def\L{\mathcal{L}}
\def\M{\mathcal{M}}
\def\Mperp{\mathcal{M}^\perp}
\def\N{\mathcal{N}}
\def\K{\mathcal{K}}
\def\F{\mathcal{F}}
\def\R{\mathcal{R}}
\def\s{\mathcal{S}}
\def\p{\mathcal{P}}
\def\P{\mathcal{P}}
\def\T{\mathcal{T}}
\def\O{\mathcal{O}}
\def\Z{\mathcal{Z}}
\def\bg{L_a^2(\mathbb{D})}
\def\hd{H^2(\mathbb{T})}
\def\phd{D^{ph}}
\def\dh{D^h}
\def\ma{\textsf{m}}
\def\cz{Calder\acute{o}n-Zygmud }
\def\l2{L^2(\dd)}
\def\sht{a space of homogeneous type}

\def\cod{\text{cod}}
\def\fd{\text{fd}}
\def\ker{\text{ker}}
\def\ran{\text{ran}}
\def\mp{\text{mp}}

\def\al{\alpha}
\def\la{\lambda}
\def\ep{\epsilon}
\def\sig{\sigma}
\def\Sig{\Sigma}
\def\cd{\mathbb{C}^d}
\def\bm{\mathcal{M}}
\def\bn{\mathcal{N}}
\def\hN{H^2\otimes \mathbb{C}^N}
\def\ba{\mathcal{A}}
\def\hm{H^2\otimes \mathbb{C}^m}
\def\mb{\mathcal{M}^{\perp}}
\def\pr{{\mathbb C}[z_1,\cdots,z_n]}
\def\nb{\mathcal{N}^{\perp}}
\def\hrd{H^2(\mathbb{D}^n)}
\def\be{\mathcal{E}}
\def\po{{\mathbb C}[z,\,w]}
\def\hr{H^2({\mathbb D}^2)}
\def\bigpa#1{\biggl( #1 \biggr)}
\def\bigbracket#1{\biggl[ #1 \biggr]}
\def\bigbrace#1{\biggl\lbrace #1 \biggr\rbrace}

\def\papa#1#2{\frac{\partial #1}{\partial #2}}
\def\dbar{\bar{\partial}}

\def\oneover#1{\frac{1}{#1}}

\def\meihua{\bigskip \noindent $\clubsuit \ $}
\def\blue#1{\textcolor[rgb]{0.00,0.00,1.00}{#1}}
\def\red#1{\textcolor[rgb]{1.00,0.00,0.00}{#1}}

\def\norm#1{||#1||}
\def\inner#1#2{\langle #1, \ #2 \rangle}

\def\divide{\bigskip \hrule \bigskip}

\def\bigno{\bigskip \noindent}
\def\medno{\medskip \noindent}
\def\smallno{\smallskip \noindent}
\def\bignobf#1{\bigskip \noindent \textbf{#1}}
\def\mednobf#1{\medskip \noindent \textbf{#1}}
\def\smallnobf#1{\smallskip \noindent \textbf{#1}}
\def\nobf#1{\noindent \textbf{#1}}
\def\nobfblue#1{\noindent \textbf{\textcolor[rgb]{0.00,0.00,1.00}{#1}}}

\def\vector#1#2{\begin{pmatrix}  #1  \\  #2 \end{pmatrix}}

\def\cfh{\mathfrak{CF}(H)}

\begin{abstract}
 \noindent This paper seeks to extend the theory of composition operators on analytic functional Hilbert spaces from analytic symbols to quasiconformal ones. The focus is the boundedness but  operator-theoretic questions are  discussed as well.  In particular, we  present a thorough analysis of $L^p$-estimates of {a class of singular integral operators} $P_\varphi$ associated with a quasiconformal mapping $\varphi$. 
\end{abstract}


\section{Introduction and Main Results}
Composition operator is  a branch of  modern operator theory  \cite{CM1995}, \cite{S1993}, \cite{SM1993}, \cite{zhu2007} and a typical object of study  involves an analytic self-map $\varphi: \dd\to \dd$ of the unit disk. Let $H$ be a Hilbert space of analytic functions, such as the Hardy space $H^2(\dd)$ or the Bergman space $L^2_a(\dd)$. Then one considers  $C_\varphi$  defined by
  $$C_\varphi(f)=f\circ \varphi, \quad f \in H.$$
  The boundedness of $C_\varphi$ follows from the Littlewood subordination principle. Once the boundedness is in hand, one can ask various operator-theoretic questions such as compactness, Schatten class membership, Fredholm theory, etc.

 \bigno
  {In this paper we seek to extend the symbol $\varphi$ to non-analytic self-maps of  $\dd$. This is  certainly not a new idea, although the theory is far from being mature.
There exist  quite a few references about  composition operators on nonanalytic function spaces, such as Sobolve-type spaces or BMO-type spaces. In these situations, the symbols are usually not analytic. Of particular interests to us  is   quasiconformal composition. Reimann's theorem on composing BMO functions with quasiconformal symbols \cite{AIM2008} is perhaps the best known result along this line.
 Other places where quasiconformal composition   appears  include
  \cite{GGR1995},  \cite{HKM2014}, \cite{Hu}, \cite{KKSS1014}, \cite{KXZZ2017}, \cite{Nag}, \cite{Schippers}, etc.
In particular, in view of the similiarity of the titles, it appears enlightening to compare the present article with  \cite{KXZZ2017} which studies  composition  on  $Q$-spaces  $Q_\alpha(\rr^n)(0<\alpha<1)$ with quasiconformal symbols whose Jacobians are $A_1$ weights.  The specific problems treated in these two papers are different.
The methodology in  \cite{KXZZ2017} is an interplay between quasiconformal mapping and harmonic analysis, whereas ours has an additional flavor of complex analysis.


 \bigno
  {In this paper we study quasiconformal composition operators on   analytic function spaces and we choose to work on the Bergman space $L_a^2(\dd)$.   Our construction clearly makes sense on other analytic spaces as well, but such a task is not pursued in this paper.
 In this setting, the composition operator is  initially defined as a map from $H^\infty(\dd)$ to $L^\infty(\dd)$. The idea of such a study can be traced back to R. Rochberg's work in 1994 \cite{rochberg1994}.  To the best of our knowledge,   there is little progress along this line since then.  In general, when acting on analytic spaces, non-analytic symbols form conceivably a rather wild world for which  the boundedness of $C_\varphi$ presents a serious challenge, not to mention other fine properties.}   One contribution of this paper is to manifest that  a decent theory can be established   for  quasiconformal composition on analytic function spaces, although the methodology is, as expected, quite different from that of analytic symbols.

 \bigno
 Our first motivation for considering quasiconformal composition starts with the observation that if one is set to consider non-analytic  $\varphi$, then  $f\circ \varphi$ is not necessarily in  the Bergman/Hardy space we start with. So we consider a Toeplitz-composition type operator:
  $$Q_\varphi(f)=P(f\circ \varphi),$$
  where $P$ denotes the Bergman/Szego projection. If we look at the   extreme situation, i.e., we ask $\varphi$ to be anti-analytic, such as $\varphi(z)=\bar{z}$,  then $Q_\varphi$ kills the analytic function $f$ entirely, except for the constant term. So   a reasonable idea is to allow $\varphi$ to have a controlled degree of analyticity, which points precisely to quasiconformal mappings.  For any fixed $K \in [1, \infty)$, we say that a homeomorphism $\varphi:\dd\to\dd$ is  a $K$-quasiconformal mapping  if it satisfies
\begin{itemize}
\item[\emph{(1)}] $\varphi\in W^{1, 2}_{\emph{loc}}(\dd)$ \emph{(}that is, $\varphi, \partial_z \varphi, \partial_{\bar{z}} \varphi \in L^2_{\emph{loc}}(\dd)$\emph{)};
\item[\emph{(2)}]  $|\partial_{\bar{z}} \varphi(z)| \leq \frac{K-1}{K+1}|\partial_z \varphi(z)|$   for almost every $z\in\dd$.
\end{itemize}

\bignobf{Two Simplifications.} (1)  It is  attempting to consider general quasiregular mappings instead of  the quasiconformal symbols  used in this paper, which are homeomorphisms of the unit disk; this will simplify many arguments.
But given that this   is the first paper along this line and quasiconformal mappings  illuminate the ideas better, and moreover, they present a theory which   is  already interesting and complicated,  it appears reasonable to us to refrain to homeomorphisms at this moment. General symbols are certainly good problems for  further work.

\bigno  (2) After some initial study of the $Q_\varphi$ operator, it becomes clear to us that, when $\varphi$ is not analytic, the two maps
  $$f \mapsto f\circ \varphi \quad \text{and} \quad f\mapsto P(f\circ \varphi)$$
  differ significantly because the latter involves the extra complicacy  of the Szego projection or the Bergman projection. Because the first map is   rich and difficult enough, in this paper we   focus on it. We plan to treat the second map in a separate work. {In \cite{rochberg1994}, using probabilistic methods, Rochberg considered certain special $\varphi$ and obtained several sufficient conditions for the boundedness of $Q_\varphi$, where $P$ is the Szego projection.}

\bigno The second motivation of this paper comes from the theory of harmonic quasiconformal mapping. In 1968, Martio \cite{Ma1968} asked whether the Possion extension of a $K$-quasisymmetric mapping on the unit circle must be a quasiconformal mapping on $\dd$. This problem is negative answered by \cite{Pa2002}. We study this question via operator theory on the analytic function spaces. With the help of Heinz's inequality \cite{Heinz1959}, if $\varphi$ is a harmonic quasiconformal mapping, then both $C_\varphi$ and $C_{\varphi^{-1}}$ are bounded from $L^2_a(\dd)$ into $L^2(\dd)$.
As an application of Theorem \ref{T:boundforC}, a harmonic quasiconformal map on $\dd$ must be bi-Lipschitz on $\tt$.


\bigno The third motivation of this paper is to study twisted Bergman projections.
When $\varphi:\dd\to\dd$ is measurable,   the study of composition operators is in a sense equivalent to the study of the following singular integral operator
\begin{equation}\label{E:soperator}
P_\varphi f(z):
=\int_{\mathbb{D}}\frac{f(w)}{(1-\varphi(z)\overline{\varphi(w)})^2}dA(w),
\end{equation}
where $dA(z)$ is the normalized Lebesgue measure on $\dd$. Note that $P_\varphi=C_\varphi C_\varphi^*$.

\bigno Similar formulas have been used by a number of researchers. In \cite{CG2006}, Cowen and Gallardo-Gutierrez obtained a description of the adjoints of composition operators. Muhly and Xia used them to study automorphisms of the Toeplitz algebra in \cite{MX1995}.

\bigno In contrast to the standard Bergman projection, in general, $P_\varphi$ is not a Calderon-Zygmund-type operator (CZO) over a space of homogeneous type.
When $\varphi$ is analytic,  the $L^2$-boundedness of $P_\varphi$ on $L^2(\dd)$ is equivalent to that on $L^2_a(\dd)$. Zhu \cite{zhu20071}, \cite{zhu2007} obtained $L^p$ boundedness of  $P_\varphi$   for $1<p<\infty$.

\bigno In this paper we obtain  a rather thorough analysis of the boundedness of $P_\varphi$ for quasiconformal $\varphi$, including both weak-$(1, 1)$ and $L^p$-estimates ($1<p\leq\infty$). The proofs form a mixture of quasiconformal mapping, harmonic analysis, and operator theory.

\bigno Our first task is to characterize the $L^2$-boundedness of $C_\varphi$. Next are some needed notations.  Let $\varphi$ be a $K$-quasiconformal mapping over the unit disk $\dd$ with $\varphi(0)=0$. It extends uniquely to a homeomorphism from the closed disk $\bar{\dd}$ to itself (Theorem 8.2, \cite{lv}). We let $\tilde{\varphi}: \tt \to \tt$    denote the boundary mapping.
Let $m_\varphi(z)=\frac{1-|z|}{1-|\varphi(z)|}$.
We also recall
  the so-called extremal distortion function $\psi_K(r)$.
  \begin{definition} \textsc{(\cite{AVM1988}, \cite{AVM1997})} For each $K\in[1,\infty), r \in [0, 1]$, we define
\begin{equation}\psi_K(r)=\sup_{|z|=r}\{|\varphi(z)|: \varphi \text{ is  $K$-quasiconformal over } \dd  \text{ with } \varphi(0)=0\}.
\end{equation}
\end{definition}
\noindent It is easy to see that $\psi_K(r)$ is a strictly increasing function from $[0, 1]$ to $[0, 1]$, with $\psi_K(0)=0$ and $\psi_K(1)=1$. It satisfies the following remarkable semigroup property \cite{lv}: For any $K_1, K_2 \in [1, \infty)$,
\begin{equation}\label{E:semigroup}\psi_{K_1K_2}=\psi_{K_1}\circ\psi_{K_2}.\end{equation} Moreover, when $K=1$, $\psi_1(r)=r$ by the Schwarz lemma.

\begin{theorem}\label{T:boundforC}
The operator $C_\varphi$ extends to be a bounded operator $C_\varphi: L^2_a(\mathbb{D})\to L^2(\mathbb{D})$ if and only if $\widetilde{\varphi}^{-1}$ is  Lipschitz on $\tt$. Moreover, in case of boundedness, if $\varphi(0)=0$, then
$$c_2\sup_{z\in\mathbb{D}} m_\varphi(z) \leq  \|C_\varphi\|_{\bg\to L^2(\dd)}\leq \frac{c_1}{1-\psi_K^2(\frac{1}{2})} \sup_{z\in\dd} m_\varphi(z),$$
where $c_1$ is a universal constant and  $c_2=c_2(K)$ is a positive constant depending only on $K$.
\end{theorem}





\bigno  Koo-Smith \cite{KS2007}, Koo-Wang \cite{KW2009}, and Wogen \cite{wogen1988} contain some general boundedness criteria for smooth symbols on analytic function spaces on the higher dimensions.

\bigno Next we consider essential boundedness. {Let $\varphi: \dd \to \dd$ be an analytic mapping. For each $w\in\dd-\varphi(0)$, define  the Nevanlinna counting function by
 $N_{\varphi,\gamma}(w)=\mathop{\sum}\limits_{z\in \varphi^{-1}(w)}(\log\frac{1}{|z|})^\gamma, $
where $\gamma=1, 2$.
Shapiro's beautiful essential norm formula \cite{S1987} for composition operator $C_\varphi$ on $\hd$ states
$\|C_\varphi\|_e^2=\mathop{\limsup}\limits_{{|w|\to 1}}\frac{N_{\varphi,1}(w)}{\log(\frac{1}{|w|})}.
$
On the Bergman space Poggi-Corradini \cite{Pietro1998} obtained
$\|C_\varphi\|_e^2=\mathop{\limsup}\limits_{{|w|\to 1}}\frac{N_{\varphi,2}(w)}{\log(\frac{1}{|w|})^2}.
$ For a quasiconformal mapping $\varphi$, $
N_\varphi(z)=\log\frac{1}{|\varphi^{-1}(z)|},  z\not=0.$ } This is a good example that although the quasiconformal assumption simplifies $N_\varphi$ considerably, the proof of the following theorem is still rather technical, hence better presented for quasiconformal only. The notation $N_\varphi$ serves mainly to illustrate the similiarity.

\begin{theorem}\label{T:compactforC}
Let $\varphi$ be a quasiconformal mapping over $\dd$ with $\varphi(0)=0$. If $\widetilde{\varphi}^{-1}$ is Lipschitz on $\tt$, then
 \begin{equation}\label{E:essentialnorm}
c_2 \mathop{\lim}\limits_{{t\to 1}}\mathop{\sup}\limits_{{|z|>t}}\frac{N_\varphi(z)}{\log\frac{1}{|z|}}\leq \|C_\varphi\|_e\leq \frac{c_1 \psi_K(\frac{1}{2})}{1-\psi_K^2(\frac{1}{2})} \mathop{\lim}\limits_{{t\to 1}}\mathop{\sup}\limits_{{|z|>t}}\frac{N_\varphi(z)}{\log\frac{1}{|z|}}.
 \end{equation}
 where $c_1$ is a universal constant and $c_2=c_2(K)$  is a positive constant depending only on $K$.
In particular, $C_\varphi$ is compact if and only if $\lim_{|z|\to 1} m_\varphi(z)=0.$
\end{theorem}

\bigno The choice of $\frac{1}{2}$ in Theorem \ref{T:boundforC} and Theorem \ref{T:compactforC} can be replaced by any number in $(0, 1)$. It appears  that optimizing on this choice still won't give us the best constant, so we refrain from doing so. Say, a crude estimate shows that one can take $c_1=147$ in the above, which is certainly not sharp.  Next it comes to the  weak $(1, 1)$ and $L^p$-boundedness of $P_\varphi$.



\begin{theorem}\label{T:extraforShplus}Let $\varphi$ be a quasiconformal mapping over $\dd$ such that $\varphi(0)=0$. If $\widetilde{\varphi}$ is bilipschitz on $\tt$, then
$P_\varphi$ satisfies the weak-$(1, 1)$ inequality: if $f\in L^1(\dd)$, then for any $\alpha>0$,
       \begin{equation*}\label{E:weak11}|\{z\in\dd: |P_\varphi f(z)|>\alpha\}|\leq \frac{c}{\alpha}\int_{\dd}|f(z)|dA(z),
       \end{equation*}
where $c$ is  a constant independent of $f$, and the leftmost $|\cdot|$ denotes the normalized Lebesgue area.
\end{theorem}

\begin{theorem}\label{T:Shexpra}
Let $\varphi:\dd\to\dd$ be a quasiconformal mapping such that $\varphi(0)=0$ and $P_\varphi$ is bounded on $L^2(\dd)$. Then
\begin{itemize} \item[\emph{(i)}] $P_\varphi$ is bounded on $L^p(\dd)$ for $1<p<\infty$,
\item[\emph{(ii)}] {$P_\varphi$ is bounded from $L^p(\dd)$ to $L^1(\dd)$ for $0<p<1$,}
\item[\emph{(iii)}] $P_\varphi$ is bounded from $L^\infty(\dd)$ to $\emph{BMO}(\dd)$.
\end{itemize}
\end{theorem}

\section{Proofs of Theorem \ref{T:boundforC} and Theorem \ref{T:compactforC}}
We first introduce Hersch-Pfluger's distortion theorem which serves as a quasiconformal version of the Schwarz lemma.  To do this, we extend the definition of the
extremal distortion function from $K\geq 1$ to $K> 0$ by letting $\psi_K(r)$ be the inverse function of $\psi_{\frac{1}{K}}$ when $0<K<1$.
%
It is easy to see that $\psi_K(r)$ is a strictly increasing function from $[0, 1]$ to $[0, 1]$, with $\psi_K(0)=0$ and $\psi_K(1)=1$.

\begin{lemma} \emph{(\cite{lv}, p.64)}\label{L:HPD}
 Let $\varphi$ be a $K$-quasiconformal mapping over $\dd$ with $\varphi(0)=0$. Then
\begin{equation}
\psi_{\frac{1}{K}}(|z|)\leq |\varphi(z)|\leq \psi_{K}(|z|), \quad\quad z\in\dd.
\end{equation}
\end{lemma}

\bigno Now we analyze the area distortion of a pseudohyperbolic disk, where is typical in operator theory, under a quasiconformal mapping. Recall that
a pseudohyperbolic disk $D^{ph}(a,r)$ centered at $a\in\mathbb{D}$ with radius $0<r<1$
is  given by
$D^{ph}(a,r)=\{z\in\mathbb{D}:|\tau_{a}(z)|<r\},$
where
$\tau_a(z)=\frac{a-z}{1-\bar{a}z}.$
It is a Euclidean disk with center
$C=\frac{1-r^2}{1-r^2|a|^2}a $ and
 radius  $R=\frac{1-|a|^2}{1-r^2|a|^2}r.$

\bigno
 The next lemma  is  of independent interests and will be used repeatedly.

\begin{lemma}\label{L:areachangephd} Let $\varphi$ be a $K$-quasiconformal mapping over $\dd$ such that $\varphi(0)=0$ and $D^{ph}(z,r)$ be a pseudohyperbolic disk for  $z\in\dd$ and $0<r<1$. Then\begin{equation}\label{s4eq05}
\psi_{\frac{1}{K}}^2(r)\frac{(1-|\varphi(z)|^2)^2}
{(1-\psi_{\frac{1}{K}}(r)^2|\varphi(z)|^2)^2}\leq |\varphi(D^{ph}(z,r))|\leq \psi_K^2(r)\frac{(1-|\varphi(z)|^2)^2}{(1-\psi_K(r)^2|\varphi(z)|^2)^2},
\end{equation}
where $|\cdot|$ denotes the normalized Lebesgue area on $\dd$.
\end{lemma}
\begin{proof} It suffices to show that for any $z\in\mathbb{D}$ and $0<r<1$,
\begin{equation}\label{E:changephd}
D^{ph}(\varphi(z),\psi_{\frac{1}{K}}(r))\subset \varphi[D^{ph}(z,r)]\subset D^{ph}(\varphi(z),\psi_K(r)).
\end{equation}
Define $\Phi(u)=\tau_{\varphi(z)}\circ \varphi\circ\tau_z(u),$ which is a $K$-quasiconformal mapping with $\Phi(0)=0$.
By Lemma \ref{L:HPD},
$|\Phi(u)|\leq\psi_K(|u|).$
Set $w=\tau_z(u)$, then $|\tau_{\varphi(z)}
(\varphi(w))|\leq\psi_K(|\tau_z^{-1}(w)|).$
Since  $w\in\phd(z, r)$ 
 and $\psi_K$ is  increasing, it follows  $\bigg|\frac{\varphi(z)-\varphi (w)}{1-\overline{\varphi(z)}\varphi(w)}\bigg|\leq\psi_K(|\frac{z-w}{1-\bar{z}w}|)\leq\psi_K(r)$,
which yields the second inclusion.

\bigno For the first inclusion, it is sufficient to show
$\varphi^{-1}[D^{ph}(\varphi(z),\psi_{\frac{1}{K}}(r)]\subset D^{ph}(z,r).$
Since $\varphi^{-1}$ is  $K$-quasiconformal, by the second inclusion of
(\ref{E:changephd}),
\begin{equation*}
\varphi^{-1}[D^{ph}(\varphi(z),\psi_{\frac{1}{K}}(r)]\subset D^{ph}(z,\psi_{K}\circ\psi_{\frac{1}{K}}(r)).
\end{equation*}Using the semigroup property of $\psi_K$, we can conclude that
$\varphi^{-1}(D^{ph}(\varphi(z),\psi_{\frac{1}{K}}(r))\subset D^{ph}(z,r)$.
The proof of Lemma \ref{L:areachangephd} is complete now.
\end{proof}

\noindent Next we introduce another gadget  needed for our proof. It will also play a role 
in Section \ref{S:Pplus}.
\begin{definition}\label{D:aphi}\emph{(\cite{AK2011}, p. 22)}
 Let $\varphi$ be a quasiconformal mapping over $\mathbb{D}$ such that $\varphi(0)=0$. For each $z\in\mathbb{D}$, let $$B_z=\bigg\{w\in\mathbb{D}:|z-w|\leq\frac{1-|z|}{2}\bigg\}.$$ We define
\begin{equation*}\label{s4eq07}
a_\varphi(z) = \exp\bigg[\frac{1}{2|B_z|}\int_{B_z}\log J(w,\varphi)dA(w)\bigg].
\end{equation*}
\end{definition}

\noindent
Two theorems of Astala and Koskela \cite{AK2011} will be needed.

\begin{lemma}\emph{(\cite{AK2011}, Lemma 2.3)}\label{F:AK01}
Let $\varphi$ be a $K$-quasiconformal mapping over $\dd$ with $\varphi(0)=0$. Then there exist positive constants $c_1$ and $c_2$ which depend only on $K$ such that for any $z\in\dd$,
\begin{equation}\label{s4eq08}
c_1 \frac{1-|\varphi(z)|}{1-|z|}\leq a_\varphi(z)\leq c_2
\frac{1-|\varphi(z)|}{1-|z|}.
\end{equation}
\end{lemma}


\begin{lemma}\emph{(\cite{AK2011}, Theorem 7.3)} \label{F:AK02} Let $\varphi$ be a quasiconformal mapping over $\dd$ with $\varphi(0)=0$. Then the boundary mapping
$\widetilde{\varphi}$ is a Lipschitz function on $\mathbb{T}$ if and only if there exists a constant $c>0$ such that $a_\varphi(z)<c$ for every $z\in\mathbb{D}$.
\end{lemma}

\noindent Now we are ready to give the proof of Theorem \ref{T:boundforC}. We actually show the following.
\begin{theorem}\label{T:boundforCC}
Let $\varphi$ be a $K$-quasiconformal mapping over the unit disk $\dd$ with $\varphi(0)=0$. Let $C_\varphi(f)=f\circ \varphi$, initially defined as a map from $H^\infty(\dd)$ to $L^\infty(\dd)$.
\begin{itemize} \item[\emph{(1)}.] The following are equivalent:
\begin{itemize}
\item[\emph{(i)}] The map $f \mapsto f\circ \varphi$ extends to a bounded operator $C_\varphi: L^2_a(\mathbb{D})\to L^2(\mathbb{D})$;
\item[\emph{(ii)}] The boundary mapping $\widetilde{\varphi}^{-1}: \tt \to \tt$ is  Lipschitz;
\item[\emph{(iii)}] $
\mathop{\sup}\limits_{z\in\mathbb{D}} m_\varphi(z)<\infty$, where
$m_\varphi(z)=\frac{1-|z|}{1-|\varphi(z)|}, \quad z \in \dd.$
\end{itemize}
\item[\emph{(2)}.] When $C_\varphi: L^2_a(\dd) \to L^2(\dd)$ is bounded, we have the following estimate of its norm:
$$ \|C_\varphi\|_{\bg\to L^2(\dd)}\leq \frac{c}{1-\psi_K^2(\frac{1}{2})} \sup_{z\in\dd} m_\varphi(z)$$
where $c$ is a universal constant.
\item[\emph{(3)}.] We also have a lower estimate $ c_1\sup_{z\in\mathbb{D}} m_\varphi(z)\leq \|C_\varphi\|_{\bg\to L^2(\dd)},$
where $c_1=c_1(K)$ is a positive constant depending only on $K$.
\end{itemize}
\end{theorem}

\bigno As hinted in the introduction, the sharp constants are unknown but certainly attractive.

\begin{proof}
For $0<r<1$, $z\in \dd$, we introduce a new function
$
\widehat{\varphi}_r(z)=\frac{|\varphi^{-1}(\phd(z,r))|}{|\phd(z,r)|}
$
and consider the Toeplitz operator $T_{\widehat{\varphi}_r}$. By Lemma \ref{L:areachangephd}, $\|T_{\widehat{\varphi}_r}\|_{\bg\to \bg}\leq 4c_1\sup_{z\in\dd}m_\varphi(z)^2,$ where $c_1=\frac{\psi_K^2(r)}{r^2(1-\psi_K(r)^2)^2}.$
We recall a basic estimate for $f\in\bg$ (\cite{DS2003}, Lemma 13).
For $0<t<1$, set $r=\frac{1+t}{2}$. Then for any $z\in\dd$, we have
$$
|f(z)|^2\leq\frac{4(1-r)^{-4}}{|\phd(z,r)|}\int_{\phd(z,r)}|f(w)|^2dA(w).
$$
Moreover, for each $w\in D^{ph}(z,r)$,
$$
\frac{(1-r|z|)^2}{(1-|z|^2)^2}\leq |\frac{1}{(1-\bar{z}w)^2}|\leq\frac{(1+r|z|)^2}{(1-|z|^2)^2}.
$$
 Let $J(w)=J(w,\varphi^{-1})$ be the Jacobian of $\varphi^{-1}$ and $T_J$ be the Toeplitz operator on $\bg$ with symbol $J$. We set $c_{2}=4(1-r)^{-4}$ and $c_{3}=(1-r)^2$. Then 

%
\begin{eqnarray*}
\langle T_{J}f,f\rangle
&\leq&  \frac{4c_{2}}{c_{3}r^2}\int_{\dd}\int_{\phd(z,r)}\frac{|f(w)|^2}{|1-z\bar{w}|^2}
dA(w)J(z, \varphi^{-1})dA(z)\\
&\leq&\frac{16c_{2}}{c_{3}(1-r^2)^2}\int_{\dd}\int_{\phd(z,r)}
\frac{|f(w)|^2}{|\phd(w,r)|}dA(w)J(z, \varphi^{-1})dA(z)\\
&=&  \frac{16c_{2}}{c_{3}(1-r^2)^2}\langle T_{\widehat{\varphi}_r}f, f\rangle.
\end{eqnarray*}

\medno It follows that
\begin{equation*}
\|T_J\|_{\bg\to \bg}\leq \frac{16c_{2}}{c_{3}(1-r^2)^2}\|T_{\widehat{\varphi}_r}\|_{\bg\to \bg}
\leq \frac{64c_{1}c_{2}}{c_{3}(1-r^2)^2}\sup_{z\in\dd}m_\varphi(z)^2,
\end{equation*}
for $\frac{1}{2}<r<1$. Since $C^*_\varphi C_\varphi=T_J$,
\begin{equation*}
\|C_\varphi\|_{\bg\to L^2(\dd)} \leq\frac{c_4}{r(1-r)^3(1-r^2)(1-\psi_{K}^2(r))}
\sup_{z\in\dd}m_\varphi(z).
\end{equation*}
Let $r$ approach  $\frac{1}{2}$, there is a constant $c$ such that
$\|C_\varphi\|_{\bg\to L^2(\dd)}\leq\frac{c}{1-\psi_K^2(\frac{1}{2})} \mathop{\sup}\limits_{z\in\dd
m_\varphi}(z).$

\bigno Next, for the lower bound, we need   a subharmonic property of quasiregular mappings, due to Iwaniec and Nolder  \cite{IN1985}, \cite{Nolder1990}, which we record below for the convenience of the reader.

\begin{lemma}\label{L:INthm01}
Let $\Omega$ be a domain in the plane, and let $z$ be the center of a cube $Q$ such that the closure $\overline{Q}$ is contained in $\Omega$. Let $\varphi:\Omega\to \mathbb{C}$ be a $K$-quasiregular mapping. Then for any $0<p<\infty$,
\begin{equation*}
|\varphi(z)|\leq c\bigg(\frac{1}{|Q|}\int_{Q}|\varphi(w)|^pdA(w)\bigg)^{\frac{1}{p}},
\end{equation*}
where the constant $c$ depends on $p$ and $K$.
\end{lemma}

\noindent For  $f\in L_a^2(\mathbb{D})$, $f\circ \varphi$ is $K$-quasiregular. By Lemma \ref{L:INthm01}, there exists a constant $c=c(K)$ such that for all $z\in\mathbb{D}$,
$
|f\circ \varphi(z)|\leq\frac{c}{{1-|z|}}\bigg(\int_{\mathbb{D}}|f\circ \varphi(w)|^2dA(w)\bigg)^{\frac{1}{2}}.
$
Now fix $z_0\in\mathbb{D}$ and choose
$f(w)=\frac{1-|\varphi(z_0)|^2}{(1-\overline{\varphi(z_0)}w)^2} $ for   $w\in\dd.$
It satisfies $\|f\|_{L^2(\dd)}=1$.
Also $C_\varphi f(z_0)=\frac{1}{1-|\varphi(z_0)|^2}$.
Thus,
\begin{equation*}
\frac{c}{1-|z_0|}\bigg(\int_{\mathbb{D}}|f\circ \varphi(w)|^2dA(w)\bigg)^{\frac{1}{2}}\geq \frac{1}{1-|\varphi(z_0)|^2}.
\end{equation*}
Thus
\begin{equation*}
\|C_\varphi\|_{\bg\to L^2(\dd)}\geq\frac{1}{2c}\frac{1-|z_0|}{1-|\varphi(z_0)|}.
\end{equation*}
It follows that
$\|C_\varphi\| \geq \frac{1}{2c} \mathop{\sup}\limits_{z\in\mathbb{D}}\frac{1-|z|}{1-|\varphi(z)|}$. Now we know  that $C_\varphi$ is bounded from the Bergman space $\bg$ into $L^2(\dd)$ if and only if $\mathop{\sup}\limits_{z\in\dd}m_\varphi(z)<\infty$. It follows $\textrm{(i)}  \Leftrightarrow \textrm{(iii)}$ in Part (1).
Observe that
$\mathop{\sup}\limits_{z\in \dd} m_\varphi(z)<\infty$
if and only if
$
\mathop{\sup}\limits_{z\in \dd}\frac{1-|\varphi^{-1}(z)|}{1-|z|}<\infty.
$
By Lemma \ref{F:AK01}, $\mathop{\sup}\limits_{z\in \dd} m_\varphi(z)<\infty$ if and only if $\mathop{\sup}\limits_{z\in\mathbb{D}}a_{\varphi^{-1}}<\infty$. From Lemma \ref{F:AK02}, it follows that $\mathop{\sup}\limits_{z\in\dd} m_\varphi(z)<\infty$ if and only if $\tilde{\varphi}^{-1}$ is Lipschitz on $\mathbb{T}$. The proof of Theorem \ref{T:boundforCC} is complete now.
\end{proof}

{\bignobf{Remarks.} (1) By  classical results (\cite{zhu2007}, p. 172) on Carleson measures on $L^p_a(\dd)$ for $0<p<\infty$, we now know that $C_\varphi:L^p_a(\cc)\to L^p(\dd)$ is bounded if and only if $\tilde{\varphi}^{-1}$ is Lipschitz on $\tt$.

\bigno (2) By Harnack's inequality and Lemma \ref{L:HPD}, one can obtain the following quasiconformal version of the Littlewood subordination principle, which is unfortunately not enough to yield the desired $L^2$-boundedness of $C_\varphi$. We record it below, without proof, for interested readers.
Let $\varphi$ be a $K$-quasiconformal over $\dd$ such that $\varphi(0)=0$, and let $f\geq 0$ be a subharmonic function on $\mathbb{D}$. Then for any $r\in(0, 1)$,
\begin{equation}
\frac{1}{2\pi}\int_0^{2\pi}f(\varphi(re^{i\theta}))d\theta\leq \frac{4}{1-\psi_K(r)} \sup_{0<t<1}\frac{1}{2\pi}\int_0^{2\pi}f(te^{i\theta})d\theta,
\end{equation}
where $\psi_K$ is the extremal distortion function.}

\bigno Next we give the proof of Theorem \ref{T:compactforC}.

\begin{proof} Recall the essential norm  $\|C_\varphi\|_e= \inf_K \|C_\varphi-\textit{K}\|_{\bg\to L^2(\dd)},$ where $\mathit{K}: \bg\to L^2(\dd)$ is compact.
We shall need an integrability criteria for the Jacobians of quasiconformal mappings (\cite{astala1994}, Corollary 1.2). Again we record it for the convenience of the reader.

\begin{lemma}\label{L:JacInt}
Let $\varphi$ be a $K$-quasiconformal mapping over $\dd$ with $\varphi(0)=0$. Then    $J(z,\varphi)\in L^q(\mathbb{D})$ for   $ q\in [0, \frac{K}{K-1})$.
\end{lemma}
\noindent We shall  take  $q=\frac{K+1}{K}$.
Next we introduce a family of compact operators associated with compact subsets of $\dd$. Let $E$ be a compact subset of $\dd$, and let $M_E f= f\chi_{E}$ be a multiplication operator on $L^2(\dd)$. We define  $X_E= M_E\cdot C_\varphi.$

\mednobf{Claim A. } $X_E$ is compact on $\bg$ for each compact subset $E$ of $\dd$.

\bigno
It is sufficient to show that if
$\{f_n\}\subset L_a^2(\mathbb{D})$ uniformly converges to $0$ on each compact subset in $\mathbb{D}$, then $\|X_E f_n\|_{L^2(\dd)}$ converges to 0 as $n \to \infty$ \cite{zhu2007}. By Lemma \ref{L:JacInt}, there is a constant $c$ depending only on $\varphi$ such that
\begin{equation*}
\int_{\mathbb{D}}|f_n\circ \varphi(z)|^2\chi_{E}(z)dA(z)\leq c \bigg [\int_{\varphi(E)}|f_n(w)|^{2(K+1)}dA(w)\bigg ]^\frac{1}{K+1}.
\end{equation*}

\bigno Then, we have $\lim_{n\to\infty}\|X_Ef_n\|^2_{L^2(\dd)}=0$.

\medno Next we introduce another family of operators $\{T_r\}_{r \in (0, 1)}$. For any $r\in(0, 1)$, let $F_r\subset\dd$ be such that $\varphi(F_r)=A_r \triangleq \{z\in\dd:|z|>r\}.$
Let $E_r=\dd-F_r$ and consider the compact operator $X_{E_r}$. Then we define $T_r$ on $\bg$ by $(T_rf)=C_\varphi f-X_{E_r} f$.

\bignobf{Claim B.} There is a constant $c$ such that $$\mathop{\inf}\limits_{0<r<1}\|T_r\|\leq \frac{c\psi_K(\frac{1}{2})}{1-\psi_K^2(\frac{1}{2})}\mathop{\lim}\limits_{t\to  1}\mathop{\sup}\limits_{|z|>t}m_\varphi(z).$$

\bigno
By a known covering lemma (\cite{axlers}, Lemma 3.5), there exists a sequence $\{\xi_n\}_{n=0}^\infty \subset\dd$ such that $\mathop{\lim}\limits_{n\to \infty}|\xi_n|=1$ and
$\mathbb{D}=\mathop{\cup}\limits_{n=0}^\infty D^{ph}(\xi_n,\frac{1}{2}).$
 Moreover, each $z\in\mathbb{D}$ is contained in at most $c_1$  pseudohyperbolic disks among $\{D^{ph}(\xi_n,\frac{3}{4})\}$.
Owing to a standard subharmonic estimate \cite{DS2003}, for $f\in L_a^2(\mathbb{D})$ and any $z\in D^{ph}(z_0,\frac{1}{2}), z_0\in\dd$, we have
\begin{equation}\label{s7eq01}
|f(z)|^2\leq \frac{c_2}{|D^{ph}(z_0,\frac{3}{4})|}
\int_{D^{ph}(z_0,\frac{3}{4})}|f(w)|^2dA(w).
\end{equation}
In order to proceed with the proof of Claim B, we need to verify the following

\bignobf{Claim C.} For any $r \in (0, 1)$ close to $1$, we can select a subsequence  $\{\xi_{n_m^{(r)}}\}_{m=0}^\infty\subset\{\xi_n\}_{n=0}^\infty$ such that for any $m\geq 0$
\begin{itemize}
  \item [(1)] $\phd(\xi_{n_m^{(r)}},\frac{1}{2})\cap A_r\not=\emptyset$;
  \item [(2)] $A_r\subset \mathop{\cup}\limits_{m=0}^\infty \phd(\xi_{n_m^{(r)}},\frac{1}{2})$;
  \item [(3)] $|\xi_{n_{m^{(r)}}}|\geq t_r$, where $t_r=\min\{\frac{-2+3r+2r^2}{4-r^2}, r\}$.
\end{itemize}


\noindent  We look at those $\xi_n$ such that $\phd(\xi_n,\frac{1}{2})\cap A_r\not=\emptyset$. We re-label these $\xi_n$ as $\{\xi_{n^{(r)}_m}\}_{m=0}^\infty$. Then (1) and  (2) are clearly satisfied. The main point is to verify (3). Clearly, if $|\xi_{n^{(r)}_m}| \ge r$, then there is nothing to prove. Next we re-label the rest of the sequence as
$\{\xi_{n^{(r)}_{m_j}}\}_{j=0}^\infty $ with $|\xi_{n^{(r)}_{m_j}}| \le r.$
%
  For any $j\geq 0$, since the pseudohyperbolic distance between each $\xi_{n_{m_j}}^{(r)}$ and $\{z:|z|=r\}$ is at most $\frac{1}{2}$, there exists a point $z_{n_{m_j}^{(r)}}'\in \tt_r$
  such that $\xi_{n_{m_j}^{(r)}}\in \phd(z_{n_{m_j}^{(r)}}', \frac{1}{2}),$ where $\tt_r\doteq \{z\in\dd:|z|=r\}$.
 Observe that for any $j\geq 0$ the modulus of the Euclidean center of $\phd(\xi_{n_{m_j}^{(r)}},\frac{1}{2})$ is equal to $\frac{3r}{4-r^2}$, and its
Euclidean radius is $\frac{2(1-r^2)}{4-r^2}.$
 Then $|\xi_{n_{m_j}^{(r)}} |\geq \frac{-2+3r+2r^2}{4-r^2}$ for any $j\geq 0$.  We have
 for any $m\geq 0$, $|\xi_{n_m^{(r)}}|\geq t_r \triangleq \min\{\frac{-2+3r+2r^2}{4-r^2}, r\}$.

\bigno
Take $f\in\bg,\|f\|_{\bg}\leq 1$, by (\ref{s7eq01}) and Lemma \ref{L:areachangephd}.
then
\begin{eqnarray*}
\int_{|z|>r}|f(z)|^2J(z,\varphi^{-1})dA(z)&\leq&\sum_{m=0}^\infty\int_{D^{ph}
(\xi_{n_m^{(r)}},\frac{1}{2})}|f(z)|^2J(z,\varphi^{-1})dA(z)\\
&\leq&\frac{c_3\psi_K^2(\frac{1}{2})}{(1-\psi_K^2(\frac{1}{2}))^2} \\
&\cdot&\sup_{|z|>t_r}\bigg[\frac{1-|\varphi^{-1}(z)|}{1-|z|}\bigg]^2
\sum_{m=0}^\infty \int_{D^{ph}(\xi_{n_m^{(r)}},\frac{3}{4})}|f_n(z)|^2dA(z)\\
&\leq& \frac{c_4 \psi_K^2(\frac{1}{2})}{(1-\psi_K^2(\frac{1}{2}))^2} \sup_{|z|>t_r}\bigg[\frac{1-|\varphi^{-1}(z)|}{1-|z|}\bigg]^2.
\end{eqnarray*}

\bigno It follows that $\|T_r\|_{\bg\to L^2(\dd)}\leq \frac{c\psi_K(\frac{1}{2})}{1-\psi_K^2(\frac{1}{2})}
\sup_{\varphi^{-1}(A_{t_r})}\frac{1-|w|}{1-|\varphi(w)|}.$
By Lemma \ref{L:HPD},
 $$\|T_r\|_{\bg\to L^2(\dd)}\leq \frac{c\psi_K(\frac{1}{2})}{1-\psi_K^2(\frac{1}{2})}
\sup_{|z|>\psi_{\frac{1}{K}}(t_r)}m_\varphi(z).$$
Observe that $t_r$ converges to 1 as $r\to 1$, so does $\varphi_{\frac{1}{K}}(t_r)\to 1$.
 In summary so far, we have
\begin{equation*}\label{s6eq01}
\|C_\varphi\|_e\leq \frac{c\psi_K(\frac{1}{2})}{1-\psi_K^2(\frac{1}{2})}\lim_{t\to 1}\sup_{|z|>t}\frac{1-|z|}{1-|\varphi(z)|}.
\end{equation*}

\bigno Next let $\{z_n\in\dd\}_{n=1}^\infty$ be any sequence such that  $|z_n|\to 1$. Then we define a sequence of functions on $\dd$ by
$
f_n(z)=\frac{1-|\varphi(z_n)|^2}{(1-\overline{\varphi(z_n)}z)^2}.
$
This is a sequence of unit vectors in $\bg$ and it weakly converges to $0$.
For any  compact operator $\textit{K}$,
$$ \mathop{\limsup}\limits_{n\to\infty}\|C_\varphi f_n-\textit{K} f_n\|_{L^2(\dd)} =  \mathop{\limsup}\limits_{n\to\infty}\|C_\varphi f_n\|_{L^2(\dd)}.$$

\medno {For any $n\geq 1$, $f_n\circ \varphi$ is $K$-quasiregular on $\dd$. By Lemma \ref{L:INthm01}, there is a constant $c$ depending only on $K$ such that}
\begin{equation*}
\frac{1-|z_n|}{1-|\varphi(z_n)|}\leq c\bigg[\int_{\mathbb{D}}|f_n\circ \varphi(w)|^2dA(w)\bigg]^{\frac{1}{2}}=
c\bigg[\int_{\dd}|f_n(z)|^2J(z,\varphi^{-1})dA(z)\bigg]^{\frac{1}{2}}.
\end{equation*}
Therefore
$\|C_\varphi\|_e\geq c_{1}  \mathop{\lim}\limits_{r\to 1}\mathop{\sup}\limits_{|z|>r}m_\varphi(z)$ for some constant
$c_{1}=c_{1}(K)$ that depends only on $K$. The proof of Theorem \ref{T:compactforC} is complete now.
\end{proof}



\section{Proof  of Theorem \ref{T:extraforShplus}}\label{S:Pplus}
In the first subsection below we prove a crucial technical lemma (Lemma \ref{L:Key01}) which is clearly of independent interests. Then we prove Theorem \ref{T:extraforShplus}. Actually, we  prove
\begin{theorem}\label{T:extraforShpluss}Let $\varphi$ be a quasiconformal mapping over $\dd$ such that $\varphi(0)=0$. Let $P^+_\varphi=\int_{\mathbb{D}}\frac{f(w)}{|1-\varphi(z)\overline{\varphi(w)}|^2}dA(w),$
If $\widetilde{\varphi}$ is bilipschitz on $\tt$, then
$P_\varphi$ satisfies the weak-$(1, 1)$ inequality: if $f\in L^1(\dd)$, then for any $\alpha>0$,
       \begin{equation}\label{E:weak11}|\{z\in\dd: |P_\varphi^+ f(z)|>\alpha\}|\leq \frac{c}{\alpha}\int_{\dd}|f(z)|dA(z),
       \end{equation}
where $c$ is  a constant independent of $f$, and the leftmost $|\cdot|$ denotes the normalized Lebesgue area.
\end{theorem}
%
%

\subsection{Carleson Boxes and Quasiconformal Mappings}
This subsection is devoted to investigating the behaviors of Carlesons boxes under a quasiconformal mapping. The main result is Lemma \ref{L:Key01}. It plays a key role in the proof of Theorem \ref{T:extraforShplus}.
Let $I\subset\tt$ be an interval. Then it induces a Carleson box $Q_I$ in $\dd$:
\begin{equation*}
Q_I=\{z\in\mathbb{D}:1-|I|\leq |z|<1,\frac{z}{|z|}\in I\},
\end{equation*}
where $|I|$ is the normalized arc length of $I$. We call the point $z_I=r_Ie^{i\theta_I}$ the center of $Q_I$, where $r_I=1-\frac{|I|}{2}$ and $\theta_I$ is the midpoint of $I$.

\begin{lemma}\label{L:Key01}
Let $\varphi$ be a quasiconformal mapping over $\dd$ such that $\varphi(0)=0$. If the boundary mapping $\widetilde{\varphi}$ is bilipschitz on $\mathbb{T}$, then there is a constant $0<t_0<1$ such that for each Carleson box $Q_I$ with $|I|<t_0$, there exists a Carleson box $Q_J$ such that
\begin{itemize}
  \item [\emph{(1)}] $\varphi^{-1}(Q_I)\subset Q_J$, and
  \item [\emph{(2)}] $\frac{|Q_J|}{|\varphi^{-1}(Q_I)|}\leq c.$
\end{itemize}
Here the constant $c$ does not depend on the choice of $Q_I$.
\end{lemma}

\begin{proof} Since $\widetilde{\varphi}$ is bilipschitz, by Theorem \ref{T:boundforC},  $C_\varphi$ and $C_{\varphi^{-1}}$ are bounded. So, there exist constants $c_{1}>0$ and $c_{2}>0$ such that
for any $z\in\dd$,
\begin{equation}\label{1207eq03}
c_{1}<\frac{1-|\varphi^{-1}(z)|}{1-|z|}<c_{2}.
\end{equation}
Next we further divide the proof of Lemma \ref{L:Key01} into two subsections.

\subsubsection{Area Change of Carleson Boxes--the Upper Bound}\label{S:areaupper}
\smallno The main task of this subsection is to verify the next claim.

\bigno \textbf{Claim.} There exists a positive number $t_0\in(0,1)$ such that for any Carleson box $Q_I$ with $|I|\leq t_0$, there is a Carleson box $Q_J$ such that
\begin{itemize}
  \item [(i)] $\varphi^{-1}(Q_I)\subset Q_J,$
  \item [(ii)] $|Q_J|\leq c|Q_I|$.
\end{itemize}
Here the constant $c$ depends only on the function $\varphi$.

\bigno To verify this claim, we need a theorem of Koslela \cite{Koskela1994}. Recall that $a_\varphi(\cdot)$ is given by Definition \ref{D:aphi}.
\begin{lemma}\emph{(\cite{Koskela1994}, Lemma 2.6)}\label{L:koskelath01}
 Let $\varphi$ be a $K$-quasiconformal mapping over $\dd$, and $\gamma\subset\mathbb{D}$ be a rectifiable arc with length $l(\gamma)\geq \emph{dist}(\gamma,\mathbb{T})$. Then
\begin{equation*}
\emph{diam}(\varphi(\gamma))\leq c\int_{\gamma}a_\varphi(z)|dz|,
\end{equation*}
where the constant $c$ depends only on $K$, and \emph{diam} denotes the diameter of the set $\varphi(\gamma)$.
\end{lemma}

\begin{proof}[Proof of the Claim] Assume that $Q_I$ is a Carleson box with $|I|=t$ and the center  $z_I$. Let $z_0$ be the center of the arc $l_1 = \{z: |z|=1-t\}\cap \overline{Q}_I$. Let $\gamma_0$ be the radial segment connecting $z_I$ and $z_0$. Consider any point $\tilde{w}\in\partial \varphi^{-1}(Q_I)$.
Let $\tilde{z}\in\partial Q_I$ be such that $\varphi^{-1}(\tilde{z})=\tilde{w}$.
Now we distinguish two cases.
If $\tilde{z}\in\partial Q_I-\mathbb{T}$, then let $
\gamma_{\tilde{z}}\subset\partial Q_I-\mathbb{T}$ be the path connecting $z_0$ and $\tilde{z}$. More precisely, let $l_2$ be the radial segment in $\partial Q_I$ that contains $\tilde{z}$, and let $z_{1}=l_1\cap l_2$. Then $\gamma_{\tilde{z}}$ is the arc in $l_1\cup l_2$ passing through $z_1$.
Next, set $$\gamma=\gamma_0+\gamma_{\tilde{z}}.$$ Observe that $l(\gamma_0)=\frac{t}{2}$ and $l(\gamma_{\tilde{z}})|\leq (\pi+1)t$, so we have $$\frac{t}{2}\leq l(\gamma)\leq (\frac{3}{2}+\pi) t<5t.$$
Since $\textrm{dist}(\gamma, \mathbb{T})\leq \textrm{dist}(z_I, \mathbb{T})=\frac{t}{2}$,
we conclude that
\begin{equation*}
5t>l(\gamma)\geq \frac{t}{2}\geq \textrm{dist}(\gamma,\mathbb{T}).
\end{equation*}
If $\tilde{z}\in\partial Q_I\cap\mathbb{T}$, then link $\tilde{z}$ and $z_I$ by a segment contained in $Q_I$ whose arc length is at most $(\pi+\frac{1}{2})t$.
By Lemma \ref{L:koskelath01}, we have
\begin{equation*}
\text{diam}(\varphi^{-1}(\gamma))\leq c\int_{\gamma}a_{\varphi^{-1}}(z)|dz|\leq 5c c_{1}t \triangleq \delta t,
\end{equation*}
where $c_{1}=\mathop{\sup}\limits_{z\in\mathbb{D}}a_{\varphi^{-1}}(z)$. Since $\varphi^{-1}$ is Lipschitz on $\mathbb{T}$, by Lemma \ref{F:AK02}, $c_{1}<\infty$. It follows that $$\textrm{dist}(\varphi^{-1}(z_I), \partial \varphi^{-1}(Q_I))\leq \delta t<\infty.$$

\bigno Note that the constant $\delta$ depends only on $\varphi^{-1}$. Let us choose a constant $t_0\in(0, 1)$ such that $$\delta t_0<\frac{1}{100}.$$ For any $0<t<t_0$,
\begin{equation}\label{E:carlesonboxQC} \varphi^{-1}(Q_I)\subset \{w\in\mathbb{D}:|w-\varphi^{-1}(z_I)|<\delta t\}.\end{equation}

 \noindent Consider a Carleson box $Q_J$ such that $|J|=10\pi\delta t$ and the middle points of $J$ and $\varphi^{-1}(I)$ coincide.
Next, we show  \begin{equation}\varphi^{-1}(Q_I)\subset Q_J.\end{equation}

\smallno If $w'\in\partial \varphi^{-1}(Q_I)$ is the center of $\varphi^{-1}(I)$, then let $z'\in\tt$ such that $\varphi^{-1}(z')=w'$.
By (\ref{E:carlesonboxQC}), for each $z\in Q_I$,
$$|\varphi^{-1}(z)-\varphi^{-1}(z')|\leq
|\varphi^{-1}(z)-\varphi^{-1}(z_I)|+|\varphi^{-1}(z_I)-\varphi^{-1}(z')|\leq 2\delta t.$$
Then
$$\varphi^{-1}(Q_I)\subset \{w\in\mathbb{D}:|w-w'|<2\delta t\}.$$
Since $\{w\in\mathbb{D}:|w-w'|<2\delta t\}\subset Q_J$, we have
$\varphi^{-1}(Q_I)\subset Q_J$.

\bigno Then it follows from
$|Q_J|= (10\pi\delta t)^2(2-10\pi\delta t)$ and
$|Q_I|=t^2(2-t)$ that we have
$$|Q_J|\leq \frac{2(10\pi\delta)^2}{2-t}|Q_I|\leq 200(\pi\delta)^2|Q_I|.$$  
Now we have verified the claim and obtained an upper bound for the area change of a Carleson box under a quasiconformal mapping.
\end{proof}

\subsubsection {Area Change of Carleson Boxes--the Lower Bound}
For the lower bound, we shall need to work with the so-called geometric definition of quasiconformal mappings. It  is based on conformal modules of quadrilaterals.

\bigno
 A quadrilateral $Q(z_1,z_2,z_3,z_4)$ is a Jordan domain $Q$ with four consecutive boundary points $\{z_1,z_2,z_3,z_4\}$ specified and with a positive orientation. The boundary arcs $\overrightarrow{z_1z_2}$ and $\overrightarrow{z_3z_4}$ are called $a$-sides and the other two are called $b$-sides. By the Riemann mapping theorem, there is a unique conformal mapping $\phi: Q\to R$ , where $R$ is a rectangle with vertices $\{0, a, a+ib, ib: a>0, b>0\}$, such that $\phi(z_1)=0, \phi(z_2)=a, \phi(z_3)=a+ib$ and $\phi(z_4)=ib$. Then the conformal module $\textrm{Mod} (Q) $ of $Q$ is defined to be $\frac{b}{a}$. We call an orientation-preserving homeomorphism on $\dd$ $K$-quasiconformal if
$$\sup_{Q}\frac{\textrm{Mod}(\varphi(Q))}{\textrm{Mod}(Q)}\leq K<\infty,
$$
where $Q$ runs over all quadrilaterals such that $\overline{Q}\subset \dd$.
It is a fundamental and remarkable result in the theory of quasiconformal mapping that the geometric definition is equivalent to the analytic definition. A proof of this equivalence can be found in \cite{ah1}, \cite{lv}.

\bigno Now we continue our estimation of Carleson boxes.  We view a Carleson box $Q_I$ as a quadrilateral by
$Q_I=Q_I(z_1,z_2,z_3,z_4)$ such that
the vertices satisfy $|z_1|=|z_4| =1-t$ and $|z_2|=|z_3|=1$. In particular, the $a$-sides of $Q_I$ are the radial segments and the other two are the $b$-sides.
 Let $\Gamma_a$ be the set of rectifiable arcs $\gamma\subset Q_I$ which connect the  $a$-sides of $Q_I$. We define $$s_a(Q_I)=\mathop{\inf}\limits_{\gamma\in \Gamma_a}l(\gamma),$$
where $l(\gamma)$ is the arc length of $\gamma$.
Similarly, we define $s_b(Q_I)$.  A direct calculation shows that
$$s_a(Q_I)=2\pi t(1-t) \quad \text{and} \quad s_b(Q_I)=t.$$

 \bigno Next we shall need Rengel's inequality \cite{lv} which says that the conformal module of a quadrilateral $Q_I$ satisfies the double inequality
\begin{equation}
\frac{(s_b(Q_I))^2}{\pi|Q_I|}\leq \textrm{Mod}(Q_I)\leq \frac{\pi |Q_I|}{(s_a(Q_I))^2}.
\end{equation}

\medno Since $|Q_I|=t^2(2-t)$, we get
\begin{equation*}
\frac{1}{\pi(2-t)}\leq \textrm{Mod}(Q_I)\leq \frac{2-t}{4\pi(1-t)^2}.
\end{equation*}
By the geometric definition,
$$
\frac{1}{K\pi(2-t)}\leq \textrm{Mod}(\varphi^{-1}(Q_I))\leq K\frac{2-t}{4\pi(1-t)^2}.
$$
By Rengel's inequality,
$$|\varphi^{-1}(Q_I)|\geq \frac{1}{\pi} \frac{(s_b(  \varphi^{-1}(Q_I)  ))^2}{\textrm{Mod}( \varphi^{-1}(Q_I) )}.$$
Let $l_{Q_I}=\partial Q_I\cap\{z:|z|=1-t\}$. For any $z'\in\varphi^{-1}(l_{Q_I})$, let $z\in l_{Q_I}$ be such that $z'=\varphi^{-1}(z)$.
 For any  $z''\in\partial\varphi^{-1}(Q_I)\cap\tt$, let $\gamma\subset \varphi^{-1}(Q_I)$ be a curve connecting $z'$ and $z''$. Then
\begin{equation*}
l(\gamma)\geq |z'-z''|\geq \textrm{dist}(z',\mathbb{T})=1-|z'|=1-|\varphi^{-1}(z)|\geq c_{1}(1-|z|)=c_{1}t.
\end{equation*}
The last inequality is due to the left side of (\ref{1207eq03}).
Thus $$s_b(\varphi^{-1}(Q_I))\geq c_{1}t,$$ and we deduce that
\begin{equation*}
|\varphi^{-1}(Q_I)|\geq \frac{1}{\pi}c_{1}^2\frac{4\pi(1-t)^2}{K(2-t)}t^2.
\end{equation*}
It follows that for any  $t_1\in(0,1)$, there is a constant $c_{2}=c(\varphi, t_1)>0$ such that for any Carleson box $Q_I$ with $|I|\leq t_1$,
\begin{equation}\label{E:arealower}
|\varphi^{-1}(Q_I)|\geq c_{2}|Q_I|.
\end{equation}

\noindent Let $0<t_0<1$ be the constant  in the upper estimate. Then for any Carleson box $Q_I$ with $|I|<t_0$, there exists a
Carleson box $Q_J$ such that $\varphi^{-1}(Q_I)\subset Q_J.$
Furthermore, by the Claim in Subsection \ref{S:areaupper} and (\ref{E:arealower}),
\begin{equation*}
\frac{|Q_J|}{|\varphi^{-1}(Q_I)|}=\frac{|Q_J|}{|Q_I|}
\frac{|Q_I|}{|\varphi^{-1}(Q_I)|}\leq c.
\end{equation*}
The proof of Lemma \ref{L:Key01} is complete now.
\end{proof}


\bigno Next, we introduce some dyadic systems over $\dd$. 
%
Let $\zz_+=\nn \cup \{0\}$. Consider the following two dyadic grids on $\tt$,
\begin{equation*}
\mathcal{D}^0=\{[\frac{2\pi m}{2^j},\frac{2\pi (m+1)}{2^j}):m\in\mathbb{Z}_+, j\in\mathbb{Z}_+, 0\leq m<2^j\}
\end{equation*}
and
\begin{equation*}
\mathcal{D}^{\frac{1}{3}}=\{[\frac{2\pi m}{2^j}+\frac{2\pi}{3},\frac{2\pi (m+1)}{2^j}+\frac{2\pi}{3}):m\in\mathbb{Z}_+, j\in\mathbb{Z}_+, 0\leq m<2^j\}.
\end{equation*}
The first appearance of shifted dyadic grids in print is probably in page 30 of \cite{Christ}. A quick way to appreciate why  shifted dyadic grids are powerful is to look at \cite{APR2013}, \cite{GJ1982} and \cite{Mei2003}. In particular, \cite{APR2013} contains a nice application to Sarason's problem on Toeplitz products.
For each $\beta\in\{0, \frac{1}{3}\}$, let $\mathcal{Q}^\beta$ denote the collection of Carleson boxes $Q_I$ with $I\in\mathcal{D}^\beta$ and we call $\mathcal{Q}^\beta$  a Carleson box system. 
\subsection{Proof of Theorem \ref{T:extraforShpluss}}
\noindent
Let $r_0 < {t_0^2}/{16},$ where
$t_0$ is the constant in Lemma \ref{L:Key01}. Then we write
$
P_\varphi^+=(P_{\varphi}^+)_{r_0}+(P_\varphi^+)_{1-r_0},
$
where
\begin{equation*}
(P_{\varphi}^+)_{r_0}f(z)=\int_{|1-\varphi(z)\overline{\varphi(w)}|\geq r_0}\frac{f(w)}{|1-\varphi(z)\overline{\varphi(w)}|^2}dA(w)
\end{equation*}
and
\begin{equation*}
(P_{\varphi}^+)_{1-r_0}f(z)=\int_{|1-\varphi(z)\overline{\varphi(w)}|<r_0}\frac{f(w)}
{|1-\varphi(z)\overline{\varphi(w)}|^2}dA(w).
\end{equation*}
Since the integral kernel of $(P_{\varphi}^+)_{r_0}$ is bounded, $(P_{\varphi}^+)_{r_0}$
is bounded on $L^p(\dd)$ for $1\leq p\leq\infty$. Next let us focus on the $L^2$-boundedness of $(P_\varphi^+)_{1-r_0}$ under the assumption that the boundary mapping $\widetilde{\varphi}$ is bilipschitz on $\mathbb{T}$.
The next   lemma  is elementary, hence proof skipped.

\begin{lemma}\label{L:Key02} For any $z, w \in \dd$, there exists  a Carleson box $Q_I$ such that $z,w\in Q_I$ and
$\frac{1}{16}|Q_I|\leq |1-z\bar{w}|^2\leq 8 |Q_I|$.
\end{lemma}

\bigno
 For any $z, w \in \dd$ such that $|1-\varphi(z)\overline{\varphi(w)}|^2\leq r_0$, by Lemma \ref{L:Key02}, there exists a Carleson box $Q_I$ such that $\varphi(z), \varphi(w)\in Q_I$ and   $$|I|^2\leq 16 |1-\varphi(z)\overline{\varphi(w)}|^2 \leq  t_0^2.$$
By Lemma \ref{L:Key01}, there exists a Carleson box $Q_J$ such that $z, w\in Q_J$ and
\begin{eqnarray*}
\frac{1}{|1-\varphi(z)\overline{\varphi(w)}|^2}\leq \frac{1}{|Q_I|}
\leq\frac{16 c_{1}}{|\varphi^{-1}(Q_I)|} \leq \frac{c_2}{|Q_J|}.
\end{eqnarray*}
Here  we have used  the Claim in Subsection \ref{S:areaupper} and  Lemma \ref{L:Key01}.
%
Now, a few known facts will help conclude the proof. By \cite{Mei2003}, for any interval $J\subset\tt$,  there exists an interval $I'\in\mathcal{D}^0\cup\mathcal{D}^\frac{1}{3}$ such that
$J\subset I'$ and $|I'|\leq 6|J|$.
 So we can  find
an interval $I'\in\mathcal{D}^0\cup\mathcal{D}^\frac{1}{3}$ such that
\begin{eqnarray}
\frac{1}{|1-\varphi(z)\overline{\varphi(w)}|^2} \leq \frac{c_3\chi_{Q_{I'}}(z)\chi_{Q_{I'}}(w)}{|Q_{I'}|}. \label{E:dyadiccontrol03}
\end{eqnarray}

\smallno Now we write the integral kernel of $(P_\varphi^+)_{1-r_0}$ as
$
\widetilde{K}(z, w)=\frac{\chi_{S}(z,w)}{|1-\varphi(z)\overline{\varphi(w)}|^2},
$
where $S=\{(z,w)\in\mathbb{D}\times\mathbb{D}:|1-\varphi(z)\overline{\varphi(w)}|<r_0\}$.
It follows that
\begin{equation}\label{E:dyadiccontrol}
\widetilde{K}(z,w)\leq c_{3}\sum_{ I \in \{\mathcal{D}^0\cup\mathcal{D}^{\frac{1}{3}}\} }\frac{\chi_{Q_{I}}(z)\chi_{Q_{I}}(w)}{|Q_{I}|}.
\end{equation}
Now we write
$K_{\beta}(z,w)=\sum_{I\in\mathcal{D}^\beta}\frac{\chi_{Q_I}(z)\chi_{Q_I}(w)}{|Q_I|}
$
 and
$\mathcal{Q}^\beta f(z)=\int_{\dd}K_\beta (z, w)f(w)dA(w).
$
Then the proof follows from the  fact  \cite{APR2013} that the weak-\emph{(1, 1)} type inequality holds for
$\mathcal{Q}^\beta$, $\beta\in\{0, \frac{1}{3}\}$.

\section{Proof of Theorem \ref{T:Shexpra}}
We shall need three different definitions of BMO functions over the unit disk, and to show that they are all equivalent. This is probably
well known to experts. Since we cannot locate a reference and it is clearly of independent interests, we present a proof for completeness. Readers familiar with BMO may skip the next subsection. 

\subsection{Three Definitions of BMO on $\dd$}

We begin with the definition of a space of homogeneous type which plays an important role in our proofs.
\begin{definition}
\emph{(\cite{CW1971}, \cite{DH2009})}\label{0828de1}
A triple $(X, \rho, \mu)$ is a space of homogeneous type if
\begin{itemize}
  \item [\emph{(i)}] $X$ is a set,
  \item [\emph{(ii)}] $\rho$ is a quasi-metric on $X$, that is,
 $\rho:X\times X\to[0,\infty]$ is such that for any $x, y, z\in X$,
  \emph{(1)} $\rho(x, y)=0 \Leftrightarrow x=y, $
\emph{(2)}  $\rho(x, y) = \rho(y, x),$ and
\emph{(3)} $  \rho(x, y) \leq  c(\rho(x, z)+\rho(z, y))$
for some constant $c$, and
  \item [\emph{(iii)}] $\mu$ is a positive measure on $X$ such that for all $x\in X$ and $r>0$, there exists a constant $c$ such that
 \emph{(a)} $ \mu(B(x,r)) < \infty, $ and
\emph{(b)} $  \mu(B(x,2r)) \leq  c\mu(B(x,r))\label{0828eq1},$
where $B(x,r)=\{y\in X:\rho(x,y)<r\}$.
\end{itemize}
\end{definition}

\bigno Let $d(z,w)=|z-w|$ be the Euclidean distance on   $\mathbb{D}$. Let $|E|$ be the normalized Lebesgue area of  $E \subset \dd$. Then the triple $(\mathbb{D},d,|\cdot|)$ is  a space of homogeneous type $\cite{CRW1976}$.

\bigno Let $\dh(=\dh(z, r)) \triangleq \{w \in \dd: |z-w|<r\}$ be a homogeneous ball in  $(\dd, d, |\cdot|)$. For a locally integrable function $f$ on $\dd$, define its average over $E_{\dh}$,
$E_{\dh}(f)=\frac{1}{|\dh|}\int_{\dh}f(z)dA(z).
$
Now we introduce  the needed BMO spaces over $\dd$. First, we say that $f\in \textrm{BMO}_{\textrm{H}}(\dd)$ if
$$\|f\|_{\textrm{BMO}_\textrm{H}} \triangleq \sup_{\dh}E_{\dh}(|f-E_{\dh}(f)|)<\infty,$$
where $\dh$ runs over all homogeneous balls in $(\mathbb{D}, d, |\cdot|)$.
If we replace $\dh$ in the above definition by Euclidean balls and Euclidean cubes contained in $\dd$, respectively, then we get $\textrm{BMO}_\textrm{B}(\dd)$ and $\textrm{BMO}_{\textrm{C}}(\dd)$.

\bigno The purpose of this subsection is to show
\begin{theorem}\label{T:BMO}There are positive constants $c_1, c_{2}, c_{3}$ and $c_{4}$ such that
$$
c_{1}\|f\|_{\emph{BMO}_\emph{H}}\leq c_{2}\|f\|_{\emph{BMO}_{\emph{C}}} \leq c_{3}\|f\|_{\emph{BMO}_{\emph{B}}}\leq c_{4}\|f\|_{\emph{BMO}_\emph{H}}.
$$
for all $f\in \emph{BMO}_{\emph{H}}(\dd)$.
\end{theorem}

\noindent Theorem \ref{T:BMO} is an immediate consequence of Lemma \ref{L:BMO1} and Lemma
\ref{L:BMO2}.
\begin{lemma}\label{L:BMO1}
The identity map
$i: \emph{BMO}_\emph{H}(\mathbb{D})\to \emph{BMO}_{\emph{C}}(\mathbb{D})$
is an isomorphism between Banach spaces.
\end{lemma}
\begin{lemma}\label{L:BMO2}
The identity map
$i: \emph{BMO}_\emph{H}(\mathbb{D})\to \emph{BMO}_{\emph{B}}(\mathbb{D})$
is an isomorphism between Banach spaces.
\end{lemma}
\smallno Since the ideas in the proofs of the above two lemmas are similar, we only prove Lemma \ref{L:BMO1}. We need the following basic fact whose proof is easy and will be skipped.
\begin{lemma}\label{L:BMO}
Fix $f \in L^1_{\emph{loc}}(\dd)$. If for any homogeneous ball $\dh \subset \dd$,  there exists a number $c$ such that
$\frac{1}{|\dh|}\int_{\dh}|f(z)-c|dA(z)\leq c,$
then $\|f\|_{\emph{BMO}_\emph{H}(\mathbb{D})}\leq 2c.$
Moreover, a similar statement holds for $\emph{BMO}_\emph{C}(\mathbb{D})$ and $\emph{BMO}_\emph{B}(\mathbb{D})$.
\end{lemma}

\begin{proof}[Proof of Lemma \ref{L:BMO1}] Let $Q=Q(z_0,2r)\subset\dd$ be a cube with center $z_0$ and with side length $2r$. Let $B=\dh(z_0, \sqrt{2}r)$ so that $Q\subset B$. Then
\begin{eqnarray*}\label{1004eq6}
\frac{1}{|Q|}\int_{Q}|f(z)-E_Bf|dA(z)&\leq& \frac{|B|}{|Q|}\frac{1}{|B|}\int_{B}|f(z)-E_Bf|dA(z)\\
&\leq& \frac{\pi}{2}\|f\|_{\textrm{BMO}_\textrm{H}(\mathbb{D})}.
\end{eqnarray*}
By Lemma \ref{L:BMO}, $\|f\|_{\textrm{BMO}_\textrm{C}(\mathbb{D})}\leq \pi\|f\|_{\textrm{BMO}_\textrm{H}(\mathbb{D})}.$
Next, we show that the mapping $i: \textrm{BMO}_{\textrm{H}}(\dd)\to \textrm{BMO}_\textrm{C}(\dd)$ is onto. To do so, we need the following extension theorem of Jones which involves quasicircles. A Jordan curve $\gamma$ in the plane is called a $K$-quasicircle if there is a $K$-quasiconformal mapping $\varphi: \cc\to\cc$ such that $\varphi(\tt)=\gamma$.
In \cite{Jones1980}, Jones proved
\begin{lemma}\label{F:Jones}
 Let $\Omega$ be a bounded planar domain whose boundary is a $K$-quasicircle.
 If $f\in \emph{BMO}_\emph{C}(\Omega)$, then there exists an $\tilde{f}\in \emph{BMO}_\emph{C}(\cc)$ such that
\begin{itemize}
\item[\emph{(i)}] $\tilde{f}(z)=f(z), z\in\Omega$, and
\item[\emph{(ii)}] $\|\tilde{f}\|_{\emph{BMO}_\emph{C}(\cc)}\leq c\|f\|_{\emph{BMO}_\emph{C}(\Omega)}$, where $c=c(K)$ is a constant depending only on $K$.
\end{itemize}
\end{lemma}

\noindent Let us continue with the proof of Lemma \ref{L:BMO1}.
If $0<r<1+|z_0|$, observe that $\dh(z_0,r)\subset Q(z_0, 2r)$, then we set $Q(z_0)=Q(z_0, 2r)$. Otherwise, since $\dh(z_0, r)=\dd$, we take $Q(z_0)=Q(0, 2)$. Then
\begin{eqnarray*}
\frac{1}{|\dh(z_0,r)|}\int_{\dh(z_0,r)}|f(z)-
E_{Q(z_0)}\tilde{f}|dA(z)
&\leq&\frac{1}{|\dh(z_0,r)|}\int_{Q(z_0)}|\tilde{f}(z)-E_{Q(z_0)}\tilde{f}|dA(z)\\
&\leq& \frac{32}{\pi}\|\tilde{f}\|_{\textrm{BMO}(\cc)}\\
&\leq&\frac{32c}{\pi}\|f\|_{\textrm{BMO}_\textrm{C}(\mathbb{D})}.
\end{eqnarray*}
 So by Lemma \ref{L:BMO},  it follows $\|f\|_{\textrm{BMO}_\textrm{H}(\mathbb{D})}\leq \frac{64c}{\pi}\|f\|_{\textrm{BMO}_\textrm{C}(\mathbb{D})}$.  The proof of Lemma \ref{L:BMO1} is complete now.
 \end{proof}

\subsection{Proof of Theorem \ref{T:Shexpra}}

\begin{proof}
Part (ii) follows from Theorem \ref{T:boundforC} and the so-called Kolmogorov inquality. 
It is based on the factorization \begin{equation}\label{E:factorizarionP}
P_\varphi=C_\varphi\circ I_\varphi,
\end{equation} where
$I_\varphi f(z)=\int_{\mathbb{D}}\frac{f(w)}{(1-z\overline{\varphi(w)})^2}dA(w).$
Since $\widetilde{\varphi}^{-1}$ is Lipschitz on $\tt$, by the remark after the proof of  Theorem \ref{T:boundforCC}, $C_\varphi: L_a^p(\dd)\to L^p(\dd)$ is bounded for all $0<p<\infty$.
Thus $\|P_\varphi (f)\|_{L^p(\dd)} \lesssim \|I_\varphi (f)\|_{L^p(\dd)}.$
Let $g(z)=f\circ\varphi^{-1}(z)J(z, \varphi^{-1})$, so $I_\varphi f(z)=P(g)(z).$ Now the  Kolmogorov inequality \cite{DHZZ} for $P$ completes the proof.

\bigno Next, we prove Part (i) and Part (iii). We shall use the   interpolation theorem regarding BMO functions on homogeneous spaces.
\begin{lemma}\textsc{(\cite{DY2005})} \label{F:interpolation}
Let $\mathcal{X}$ be a space of homogeneous type. Let $1<p<\infty$.
  If a linear operator $T$ is bounded on $L^p(\mathcal{X})$, and if it is bounded from $L^\infty(\mathcal{X})$ to $\emph{BMO}_\emph{H}(\mathcal{X})$, then $T:L^q(\mathcal{X})\to L^q(\mathcal{X})$ is bounded for all $p<q<\infty$.
\end{lemma}

\smallno By Theorem $\ref{T:boundforC}$, $C_\varphi$ is bounded from $\bg$ into $L^2(\dd)$, hence  $P_\varphi$   bounded on $L^2(\dd)$. Now by standard interpolation and duality arguments, it is sufficient to  prove that $P_\varphi: L^\infty(\dd)\to \textrm{BMO}_\textrm{H}(\dd)$ is bounded.
We need the following   special case of Reimann's theorem \cite{Reimann1974}.
\begin{lemma} \label{L:Reimann}Let $\varphi$ be a quasiconformal mapping over $\dd$ with $\varphi(0)=0$. If $f\in \emph{BMO}_{\emph{C}}(\dd)$, then
$$\|f\circ \varphi\|_{\emph{BMO}_\emph{C}}\leq c\|f\|_{\emph{BMO}_\emph{C}},$$
where $c=c(K)$ depends only on $K$.
\end{lemma}

\noindent So by Reimann's theorem Theorem \ref{T:BMO}, we only need to show that $I_\varphi: L^\infty(\dd)\to \textrm{BMO}_\textrm{H}(\dd)$ is bounded.  We recall the Bloch space consists of holomorphic functions on $\dd$ such that $\|f\|_{\beta}=\sup_{z\in\dd} (1-|z|^2)|f'(z)|<\infty.$
Let $\textrm{BMOA}(\dd)=\textrm{Hol}(\dd)\cap \textrm{BMO}_{\textrm{H}}(\dd)$.
A  theorem of Coifman-Rochberg-Weiss (\cite{CRW1976}, p. 632) states that  for each $f\in \mathcal{B},\|f\|_{\beta}\approx\|f\|_{\emph{BMOA}(\dd)}$.
 So the proof of Part (i)  is reduced to show that $\textrm{I}_\varphi:L^\infty(\mathbb{D})\to \mathcal{B}$ is bounded.

\bigno Observe that for any $f\in L^\infty(\dd)$, $I_\varphi f$ is analytic on $\dd$. For any $z\in\dd$,
\begin{eqnarray*}\label{s5eq01}
|(I_\varphi f)'(z)|
&\leq&2\int_{\mathbb{D}}\frac{|f(w)\varphi (w)|}{|1-\overline{z}\varphi(w)|^3}dA(w)\\
& \leq &
2\|f\|_{L^\infty} \|C_\varphi\|_{L_a^3(\dd)\to L^3(\dd)}^3 \int_{\dd}\frac{1}{|1-z\bar{w}|^3}dA(w).
\end{eqnarray*}
 Observe that $\|C_\varphi\|_{L_a^3(\dd)\to L^3(\dd)}^3<\infty$ by the remark after the proof of Theorem \ref{T:boundforCC}.
Now we can conclude the proof  by a standard estimate (\cite{zhu2007}, Lemma 3.10), which is recorded below for  the reader's convenience.
\begin{lemma}\label{L:integralestimates}
Let $z\in \mathbb{D}, c>0$, $t>-1$, and
$
I_{c,t}(z)=\int_{\mathbb{D}}\frac{(1-|w|^2)^t}{|1-z\bar{w}|^{2+t+c}}dA(w).
$
Then
$
I_{c,t}(z)\sim\frac{1}{(1-|z|^2)^c} \  \textrm{as} \ |z|\to 1^{-}.
$
\end{lemma}


\end{proof}

\section{Concluding Remarks: Operator-theoretic Questions}\label{Subs:OTquestions}
In this section we discuss some questions which are part of the focus for  composition operators when the symbol $\varphi: \dd \to \dd$ is analytic; in particular, we discuss the characterization of compactness,   Schatten class membership, and conditions for closed ranges.
They are relatively easy now after we have laid the ground work in the previous sections.
The real challenge on these operator-theoretic properties probably only arises when one deals  with quasiregular symbols in the future.
For this reason the present section discusses some results, without proofs, to illustrate  the parallelism between quasiconformal and analytic symbols, which suggests that the quasiconformal  extension is a feasible idea  since a decent theory can be expected.
The proofs should present no serious difficulty for anyone reasonably skilled in analytic composition operator theory.
{By Theorem \ref{T:compactforC}, $C_\varphi:L^p_a(\mathbb{D})\to L^p(\mathbb{D})$ is compact for some $0<p<\infty$ if and only if $\mathop{\lim}\limits_{|z|\to 1} m_\varphi(z)=0.$ Now we let $\mathcal{S}_p$  ($0<p<\infty$)  denote the Schatten class of bounded operators between $L_a^2(\mathbb{D})$ and $L^2(\mathbb{D})$. Using  ideas of Luecking and Zhu (Theorem 3, \cite{LZ1992}) in analytic settings, one can prove  that   the quasiconformal composition operator $C_\varphi$   on $L_a^2$ is in $\mathcal{S}_p$ if and only if for some $r \in (0, 1)$, we have
$$
 \int_{r<|z|<1}\bigg(\frac{N_\varphi(z)}{\log\frac{1}{|z|}}\bigg)^pd\lambda(z)<\infty,
$$ where $d\lambda(z)$ is the Mobius-invariant measure given by
$
d\lambda(z)=\frac{1}{(1-|z|^2)^2}dA(z).
$
Next we consider when $C_\varphi$ has a closed range. In other words, we are concerned with when $C_\varphi$ is bounded below. With an application of  reverse Carleson measures on the Bergman space and Lemma \ref{L:areachangephd}, one can show that $C_\varphi$ is bounded below on $L^p_a(\dd)$ if and only if $\widetilde{\varphi}$ is Lipschitz on $\mathbb{T}$. This should be compared with the classical result of Akeroyd and Ghatage \cite{AG}.} Moreover, Luecking-type conditions \cite{Luecking1981}, \cite{Luecking1985} can  be used to provide sufficient conditions for closed ranges.

\bigno \textbf{Acknowledgement}

\bigno X. Fang  is supported by MOST of Taiwan (106-2115-M-008 -001 -MY2) and NSFC 11571248 during his visit to Soochow University in China.
K. Guo is supported by NSFC (11371096).
Z. Wang is supported by NSFC (11601296)
, NSF of Shaanxi (2017JQ1008) and MOST 105-2811-M-008-029 during his visit to National Central University in Taiwan.

\bigno
The project of using quasiconformal symbols for composition operators on the Hardy or Bergman spaces was first conceived (for us) in 2011 when the first author was still at Kansas State University, after several discussions with his colleague (Pietro Poggi-Corradini), to whom we are most deeply grateful. The first draft of this paper was completed and reported in the 2014 Chongqing Analysis Meeting. We are grateful to Richard Rochberg for bringing his paper \cite{rochberg1994} to our attention after our talk and giving us other suggestions. Several drafts of this paper have been circulated in the last few years and the current version is considerably shortened.

\bigskip

\author{X. Fang, Department of Mathematics, National Central University, Chung-Li 32001, Taiwan; Email, xfang@math.ncu.edu.tw}

\bigskip

\author{K. Guo, School of Mathematical Sciences, Fudan University, Shanghai 200433, P. R. China; Email, kyguo@fudan.edu.cn}

\bigskip

\author{Z. Wang, School of Mathematical Sciences, Shaanxi Normal University, Xi'an 710062, P. R. China, Email, zipengwang@snnu.edu.cn}

\end{document}